\documentclass[12pt,oneside]{amsart}
\usepackage[normalem]{ulem} % for \sout strikeout text
\usepackage{latexsym,amssymb,amsmath,graphicx,hyperref}
\usepackage{subfig}
\usepackage{caption}
\usepackage{verbatim}
\numberwithin{equation}{section}
\usepackage{tikz}
\usetikzlibrary{shapes}

\usepackage{fp}
\usepackage{rotating}
\usetikzlibrary{arrows}
\usepackage{longtable}
\usepackage{mathtools}
\usepackage{centernot}

\DeclarePairedDelimiter\floor{\lfloor}{\rfloor}

\newtheorem{theorem}{Theorem}[section]
\newtheorem{lemma}[theorem]{Lemma}
\newtheorem{corollary}[theorem]{Corollary}

\newtheorem{question}[theorem]{Question}

\theoremstyle{definition}
\newtheorem{definition}[theorem]{Definition}
\newtheorem{remark}[theorem]{Remark}
\newtheorem{example}[theorem]{Example}

\newcommand{\PIVAL}{3.14159265358979323846264338} %There is no easily accessible constant
\newcounter{i} % Counter for the while loop
\newcommand{\Circulant}[2] { %Accepts 2 argument, the number of nodes, and the list of vertices
	\begin{tikzpicture}
	\setcounter{i}{0}
	\whiledo{\value{i}<#1}{ %Start counting through the nodes
		\FPmul\tempA{2}{\thei} %For the X,Y Formulae start with 2*i
		\FPdiv\tempB{\PIVAL}{#1} % pi/(nodes)
		\FPmul\tempC{\tempA}{\tempB} % the product of the previous
		\FPcos\varX{\tempC} % cos that for x
		\FPsin\varY{\tempC} % sin it for y
		\stepcounter{i} % Count up 1
		\FPround\varX{\varX}{3}
		\FPround\varY{\varY}{3}
		\node (\thei) at (\varX,\varY)[place]{ }; % Draw the nodes
		\foreach \x in {#2} { % For all the variables in the Vertice list
			\pgfmathparse{mod(\x+\thei,#1)} % Do the modular count forward
			\let\tempB\pgfmathresult
			\pgfmathparse{mod(\thei-\x,#1)} % And the modular count backward
			\let\tempA\pgfmathresult
			\ifthenelse{\lengthtest{\tempA pt < 1 pt}}{\FPadd\tempA{\tempA}{#1}}{}
			\ifthenelse{\lengthtest{\tempB pt < 1 pt}}{\FPadd\tempB{\tempB}{#1}}{}
			\ifthenelse{\lengthtest{\tempA pt > \thei pt}}{}{\ifthenelse{\thei = \tempA}{}{\draw [] (\thei) to (\tempA)}};
			\ifthenelse{\lengthtest{\tempB pt > \thei pt}}{}{\ifthenelse{\thei = \tempB}{}{\draw [] (\thei) to (\tempB)}};
		}
	}
	\end{tikzpicture}
}

\begin{document}

\tikzstyle{place}=[draw,circle,minimum size=0.5mm,inner sep=1pt,outer sep=-1.1pt,fill=black]

%%%%%%%%%%%%%%%%%%%%%%%%%%%%%%%%%%%%%%%%%%%%%%%%%%%%%%%%%%%%%%%%%%%%%%%%

\title{Independence complexes of well-covered circulant graphs}
\thanks{Last updated: \today}
\thanks{Research of the last two authors supported in part by NSERC Discovery Grants.
Research of the first author supported by an NSERC USRA}

\author{Jonathan Earl}
\address{Department of Mathematics\\
Redeemer University College, Ancaster, ON, L9K 1J4, Canada}
\email{jearl@redeemer.ca}

\author{Kevin N. Vander Meulen}
\address{Department of Mathematics\\
Redeemer University College, Ancaster, ON, L9K 1J4, Canada}
\email{kvanderm@redeemer.ca}

\author{Adam Van Tuyl}
\address{Department of Mathematics \& Statistics\\
McMaster University \\
Hamilton, ON, L8S 4L8, Canada}
\email{vantuyl@math.mcmaster.ca}


\begin{abstract}
We study the independence complexes of families of well-covered 
circulant graphs discovered by  Boros-Gurvich-Milani\v{c}, 
Brown-Hoshino,  and Moussi.  Because these graphs are well-covered, 
their independence complexes are pure simplicial complexes.  We determine 
when these pure complexes have extra combinatorial (e.g. vertex decomposable, 
shellable) or topological (e.g. Cohen-Macaulay, Buchsbaum) structure. 
We also provide a table of all well-covered circulant graphs on 16 or less vertices, 
and for each such graph, determine if it is vertex decomposable, shellable, 
Cohen-Macaulay, and/or Buchsbaum.  A highlight of this search
is an example of a graph whose independence complex is shellable but
not vertex decomposable.
\end{abstract}

\maketitle

%%%%%%%%%%%%%%%%%%%%%%%%%%%%%%%%%%%%%%%%%%%%%%%%%%%%%%%%%%%%%%%%%%%%%%%%%%%%%%%
\section{Introduction}\label{sec:intro}

Let $G$ be a finite simple graph with vertex set $V$ and edge set $E$.  We say that 
a subset $W \subseteq V$ is a {\it vertex cover} of $G$ if $e \cap W \neq 
\emptyset$ for all edges $e \in E$.  The complement of a vertex cover is an
{\it independent set}.  A graph $G$ is called {\it well-covered} if
every minimal vertex cover (with respect to the partial order of inclusion) has the 
same cardinality.  Via the duality between vertex covers and independent
sets, being well-covered is equivalent to the property that every maximal independent
set has the same cardinality.   
Plummer's survey \cite{P} provides a nice entry point to learn more about
well-covered graphs.

Recently, there has been interest in identifying circulant graphs that are 
well-covered (see, e.g. \cite{BGM,BHmusic,Brown11,HPhD,M} and some motivation there-within). 
Recall that a {\it circulant graph} is defined as follows.  Let $n 
\geq 1$ be an integer, and let $S \subseteq \{1,2,\ldots,\floor{\frac{n}{2}}\}$. The circulant graph 
$C_n(S)$ is the graph on the vertex set $V = \{0, 1,\ldots,n-1\}$, such that 
$\{a, b\}$ is an edge of $C_n(S)$ if and only if $|a-b| \in S$ or $n-|a-b| \in S$. Circulant
graphs include the family of cycles ($C_n = C_n(\{1\})$) and the family of complete
graphs
($K_n = C_n(\{1,2,\ldots,\floor{\frac{n}{2}}\})$).  
For ease of notation, we will
write $C_n(a_1,\ldots,a_t)$ instead of $C_n(\{a_1,\ldots,a_t\})$.  The circulant
graph $C_{13}(1,3,5)$ can be found in Figure \ref{c13}.  
\begin{figure}[h]
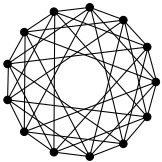

\[\Circulant{13}{1,3,5}\]
\caption{The circulant graph $C_{13}(1,3,5)$.}
\label{c13}
\end{figure}

If we consider the {\it independence complex of a graph $G$}, i.e., the simplicial
complex 
\[{\rm Ind}(G) = \{W \subseteq V ~|~ \mbox{$W$ is an independent set of $G$}\},\] 
then the well-coveredness of $G$ is equivalent to the property that ${\rm Ind}(G)$ is
a pure simplicial complex, that is, all of its maximal faces have the same size.  
A pure simplicial complex can have additional combinatorial or topological structure; in fact,
as summarized in Theorem \ref{facts},
we have the following hierarchy (definitions will be postponed to the next section)
of pure simplicial complexes:
\[\mbox{vertex decomposable} \Rightarrow \mbox{shellable} \Rightarrow
\mbox{Cohen-Macaulay} \Rightarrow \mbox{Buchsbaum}. \]
Thus, given a well-covered circulant graph $C_n(S)$, we can ask what additional
structure ${\rm Ind}(C_n(S))$ entertains.  This question is the main focus of this paper.
In particular, we look at specific families of well-covered circulant graphs
found in \cite{BGM,Brown11,M} and examine the structure of the corresponding independence complex.
This paper can be seen as a direct sequel to \cite{VVW} 
which considered the same question for some families of well-covered circulants found in \cite{Brown11}.  
In addition, our work complements the investigation of the topology
of independence complex of circulants (e.g., see \cite{A,K}). 

We now provide an overview of the paper. 
Section 2 provides the relevant background material.
In Section 3 we consider the well-covered circulants of the form 
$C_n(d+1,d+2,\ldots,\floor{\frac{n}{2}})$ which were characterized 
by Brown and Hoshino \cite{Brown11}.  Theorem \ref{thmcomplements} shows
that the independence complexes of these graphs are always Buchsbaum.  In addition, we classify 
when these complexes are vertex decomposable, solving an open problem
of \cite{VVW}.  We also
include a discussion on the $h$-vectors of Buchsbaum complexes.  In Section 4
we consider circulants of the form $C_n(S)$ where $|S| = \floor{\frac{n}{2}}-1$.
The well-covered circulants of this form where characterized 
by Moussi \cite{M}.  Theorem \ref{theorem62} 
refines this result by describing 
the additional structure of ${\rm Ind}(C_n(S))$.  
In Section 5, we examine one-paired circulants, a family of graphs
introduced by Boros, Gurvich, and Milani\v{c} \cite{BGM} as an example
of CIS (Cliques Intersecting Stable sets) circulants.  We give a new
structural result for one-paired circulants (Theorem \ref{structure2}), and use 
this result to determine some properties of its associated independence complex, 
and to provide a new proof that these circulant graphs are CIS.  

In the final two sections we collect together a number of observations
and questions based upon a computer search.
In particular,
we include a table (Table \ref{table}) of all well-covered circulants $C_n(S)$ with $n \leq 16$, and 
for each circulant, determine whether it is vertex decomposable, 
shellable, etc.  As an interesting by-product of this search, we have found that 
$C_{16}(1,4,8)$ is the smallest example of a circulant graph
that has a shellable independence complex but is not vertex decomposable.
To the best of our knowledge, $C_{16}(1,4,8)$ is the first example of a
graph with this property (see Remark \ref{vdnotshellable}).

Although we will not employ this view point, there is an 
algebraic interpretation of our work.  Associated
to a graph $G$ on $n$ vertices $V = \{x_1,\ldots,x_n\}$ 
is a quadratic square-free monomial ideal $I(G) = \langle x_ix_j ~|~ \{x_i,x_j\} \in E \rangle$ 
in $R =k[x_1,\ldots,x_n]$.
The ideal $I(G)$, commonly called the {\it edge ideal} of $G$, is the 
ideal associated to ${\rm Ind}(G)$ via the Stanley-Reisner correspondence.  
The property that $G$ is well-covered is equivalent to the property that
$I(G)$ is an unmixed ideal.  Moreover, if ${\rm Ind}(G)$ is Cohen-Macaulay 
or Buchsbaum, the ring $R/I(G)$
also has this property.  By identifying Cohen-Macaulay and/or Buchsbaum 
independence complexes, we are contributing to an ongoing programme in 
combinatorial commutative algebra to identify graphs $G$ such that
$R/I(G)$ is Cohen-Macaulay or Buchsbaum 
(e.g., see \cite{CN,HHbook,HH,HHZ,HHKO,MMCRTY,MV,V}).

Finally, we would like to add a small erratum to \cite{VVW}.  
On page 1902, the $f$-vector and $h$-vector
on line 4 should be $(1,11,33,22)$, respectively, $(1,8,14,-1)$;  
the conclusions are still the same.  In Example 6.3, it should read
$G[H] = C_{10}(1,2,3,5)$ and $H[G] = C_{10}(1,4,5)$, that is, $G[H]$ and $H[G]$ should be reversed.
In addition,
unknown to us at the time, Theorem 4.65 of Hoshino's Ph.D. thesis \cite{HPhD}  
contained a proof for the equivalence of Theorem 3.4 $(ii)$ and $(iv)$ of \cite{VVW}.  Moreover, 
we note that Theorem 3.7 of \cite{VVW} answered 
Conjecture 4.68 of \cite{HPhD}.  

\noindent
{\bf Acknowledgements.}
We used the \LaTeX\ code of \cite{eastman} to 
draw circulant graphs.  We also thank
Russ Woodroofe for answering some of our questions.

%%%%%%%%%%%%%%%%%%%%%%%%%%%%%%%%%%%%%%%%%%%%%%%%%%%%%%%%%%%%%%%%%%%%%%%%%%%%%%%%

\section{Background Definitions and Results}\label{sec:background} 

In this section we review the relevant definitions and results.  

A {\it simplicial complex} on a vertex set $V = \{x_1,\ldots,x_n\}$ is a set $\Delta$ whose elements
are subsets of $V$ such that $(a)$ if $F \in \Delta$ and $G \subseteq F$, then $G \in \Delta$, and
$(b)$ for each $i=1,\ldots,n$, $\{x_i\} \in \Delta$.   Note that the set $\emptyset \in \Delta$.
The independence complex
${\rm Ind}(G)$, as defined in the introduction, is a simplicial complex.

An element $F \in \Delta$ is called a {\it face}.  The maximal elements of $\Delta$, with respect to
inclusion, are called the {\it facets} of $\Delta$.  If $\{F_1,\ldots,F_t\}$ is a complete
list of the facets of $\Delta$, we will sometimes write $\Delta = \langle F_1,\ldots,F_t\rangle$.
The {\it dimension} of a face $F \in \Delta$,
denoted $\dim F$, is given by $\dim F = |F|-1$, where we make the convention
that $\dim \emptyset = -1$.  
The {\it dimension of $\Delta$}, denoted
$\dim \Delta$, is defined to be $\dim \Delta = \max_{F \in \Delta} \{\dim F\}$.  A simplicial complex is {\it pure}
if all of its facets have the same dimension.  Note that if $\alpha(G)$ denotes
the cardinality of the largest independent set, then $\dim {\rm Ind}(G) = \alpha(G)-1$.

The  $f$-vector of $\Delta$ records the number of faces of dimension $i$ of $\Delta$.  Precisely,
if $\dim \Delta = d$, then the {\it $f$-vector of $\Delta$} is a $(d+2)$-tuple $f(\Delta) = (f_{-1},f_0,f_1,\ldots,f_d)$
where $f_i$ is the number of faces of dimension $i$.  It follows that $f_{-1} = 1$ and $f_0 = n$.  
The {\it $h$-vector} of a $d$-dimensional simplicial complex
$\Delta$ is a $(d+2)$-tuple $h(\Delta) = (h_0,h_1,\ldots,h_{d+1})$ where
\begin{equation}\label{hvector}
h_i = \sum_{j=0}^i(-1)^{i-j}\binom{d+1-j}{i-j}f_{j-1}.
\end{equation}

Given any face $F \in \Delta$, we define the {\it link} of $F$ in $\Delta$ to be
\[{\rm link}_\Delta(F) = \{G \in \Delta ~|~ F \cap G = \emptyset ~\mbox{and}~ F \cup G \in \Delta\}.\]
The {\it deletion} of a face $F$ in $\Delta$ is the set 
\[{\rm del}_{\Delta}(F) = \{G \in \Delta ~|~ F \not\subseteq G \}.\]
Both ${\rm link}_\Delta(F)$ and ${\rm del}_\Delta(F)$ are simplicial complexes.  When $F = \{x_i\}$,
then we simply write ${\rm link}_\Delta(x_i)$ or ${\rm del}_\Delta(x_i)$.

As promised in the introduction, we now define the relevant pure simplicial complexes.  

\begin{definition} Let $\Delta = \langle F_1,\ldots,F_t\rangle$ 
be a pure simplicial complex on $V = \{x_1,\ldots,x_n\}$.  
\begin{enumerate}
\item[$(i)$] (see \cite{PB}) $\Delta$ is {\it vertex decomposable}
if $(a)$ $\Delta$ is a simplex, i.e., $\{x_1,\ldots,x_n\}$ is the unique maximal facet, or $(b)$ there
exists a vertex $x$ such that ${\rm link}_{\Delta}(x)$ and ${\rm del}_\Delta(x)$ are vertex decomposable.
\item [$(ii)$]$\Delta$ is {\it shellable} if there exists an ordering
$F_1 < F_2 < \cdots < F_t$ such that for all $1 \leq j < i \leq t$, there is some 
$x \in F_i \setminus F_j$ and some $k \in \{1,\ldots,i-1\}$ such that $\{x\} = F_i \setminus F_k$.
\item[$(iii)$] 
$\Delta$ is {\it Cohen-Macaulay}
if for all $F \in \Delta$, $\tilde{H}_i({\rm link}_\Delta(F),k) = 0$ for all $i < \dim {\rm link}_\Delta(F)$.  
(Here $\tilde{H}_i(-,k)$ denotes the $i$-th reduced simplicial homology group.)
\item[$(iv)$] $\Delta$ is {\it Buchsbaum} if ${\rm link}_{\Delta}(x)$ is Cohen-Macaulay
for all $x \in V$.
\end{enumerate}
\end{definition}

\begin{remark}
 Note that for the definition of Cohen-Macaulay, we use the reduced 
simplicial homology definition due to Reisner,
sometimes known as Reisner's Criterion.  For an introduction
to reduced simplicial homology, see \cite[Section 5.2]{Vbook}
and  for Reisner's Criterion, see \cite[Theorem 5.3.5]{Vbook}.
If $I_{\Delta}$
is the Stanley-Reisner ideal (see, e.g. \cite{Vbook}) associated to $\Delta$, 
then one can show that the definitions of Cohen-Macaulay and Buchsbaum presented here
are equivalent to the algebraic requirement that $R/I_{\Delta}$ be either 
Cohen-Macaulay or Buchsbaum,
where $R = k[x_1,\ldots,x_n]$. 
\end{remark}

We summarize a number of well-known results about the above simplicial
complexes.

\begin{theorem}\label{facts}
Let $\Delta$ be a pure simplicial complex on $V = \{x_1,\ldots,x_n\}$.  
\begin{enumerate}
\item[$(i)$] The following implications hold:
\[\mbox{vertex decomposable} \Rightarrow \mbox{shellable} \Rightarrow \mbox{Cohen-Macaulay}
\Rightarrow \mbox{Buchsbaum}.\]
\item[$(ii)$] If $\dim \Delta = 0$, then $\Delta$ is vertex decomposable 
(and thus, shellable, Cohen-Macaulay, and Buchsbaum).
\item[$(iii)$] If $\dim \Delta =1$, then $\Delta$ is vertex decomposable/shellable/Cohen-Macaulay 
if and only if $\Delta$ is connected.  If $\Delta$ is not connected, then $\Delta$ is 
Buchsbaum but not Cohen-Macaulay.
\item[$(iv)$] If $\dim \Delta \geq 2$ and $\Delta$ is Cohen-Macaulay, then $\Delta$ is connected.
\item[$(v)$] If $\Delta$ is Cohen-Macaulay, then $h(\Delta)$ has no negative entries.
\end{enumerate}
\end{theorem}

\begin{proof}
$(i)$ Vertex decomposability implies shellability by \cite[Corollary 2.9]{PB};
shellability implies Cohen-Macaulay by \cite[Theorem 5.3.18]{Vbook},
and Cohen-Macaulay implies Buchsbaum since ${\rm link}_{\Delta}(F)$ is Cohen-Macaulay
for all faces $F \in \Delta$ when $\Delta$ is Cohen-Macaulay by \cite[Proposition 5.3.8]{Vbook}.
$(ii)$ is 
\cite[Proposition 3.1.1]{PB}, and $(iii)$ is \cite[Theorem 3.1.2]{PB}.
Because ${\rm link}_\Delta(\emptyset)  = \Delta$,  $\tilde{H}_0(\Delta,k) = 0$ since
$\dim \Delta \geq 2$ and $\Delta$ is Cohen-Macaulay.  By \cite[Proposition 5.2.3]{Vbook},
 $\tilde{H}_0(\Delta,k)+1$ is the number of connected components of $\Delta$, i.e.,
$\Delta$ is connected.
For $(v)$ see \cite[Theorem 5.4.8]{Vbook}.
\end{proof}

Because we are interested in ${\rm Ind}(G)$ for graphs $G$, we will say $G$ is 
vertex decomposable, shellable, Cohen-Macaulay, or Buchsbaum if ${\rm Ind}(G)$ has the 
corresponding property. Except for the Buchsbaum property, it is enough to 
consider connected components of $G$.

\begin{lemma}\label{connected}
Suppose that $G$ and $H$ are two disjoint graphs that are both
vertex decompos-able/shellable/Cohen-Macaulay. Then $G \cup H$ is 
vertex decomposable/shellable/Cohen-Macaulay. 
\end{lemma}

\begin{proof}
See \cite[Lemma 20]{W} which states a more general result (i.e., 
for non-pure vertex decomposability and non-pure shellability).
The Cohen-Macaulay case can also be found in \cite[Proposition 6.2.8]{Vbook}.
\end{proof}

The union of two or more Buchsbaum graphs will fail to be Buchsbaum.

\begin{lemma}\label{buchsbaumunion}
Let $G$ and $H$ be two disjoint graphs that are both Buchsbaum, but
not Cohen-Macaulay.  Then
$G \cup H$ is not Buchsbaum.
\end{lemma}

\begin{proof}
Let ${\rm Ind}(G)$ and ${\rm Ind}(H)$ be the independence complexes
associated to $G$ and $H$.  The join of these two simplicial
complexes gives us the independence complex of $G \cup H$:
\[{\rm Ind}(G \cup H) = {\rm Ind}(G)\star {\rm Ind}(H) = \{F \cup E ~|~ 
F \in {\rm Ind}(G)~\mbox{and}~ E \in {\rm Ind}(H) \}.\]
Let $x$ be any vertex of $G \cup H$.  Without loss of generality, assume
that $x$ is in $G$.  Then
\[{\rm link}_{{\rm Ind}(G \cup H)}(x) = {\rm link}_{{\rm Ind}(G)}(x) \star {\rm Ind}(H).\]
By \cite[Proposition 5.3.16]{Vbook}, the join of two simplicial complexes is Cohen-Macaulay
if and only if both complexes are Cohen-Macaulay.  But ${\rm Ind}(H)$ is not Cohen-Macaulay,
so ${\rm link}_{{\rm Ind}(G \cup H)}(x)$ cannot be Cohen-Macaulay, and thus, $G \cup H$ cannot
be Buchsbaum.
\end{proof}

\begin{example}\label{example1}
Consider the four-cycle $G = C_4$, that is, the circulant graph $C_4(1)$:
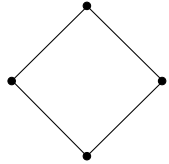
\[\Circulant{4}{1}\]
If we label the vertices $0,1,2,3$ in clockwise order,
then 
${\rm Ind}(G) = \langle \{0,2\},\{1,3\} \rangle$ is disconnected.
Since  $\dim {\rm Ind}(G) =1$, by Theorem \ref{facts} $(iii)$ that $C_4(1)$ is Buchsbaum, 
but not Cohen-Macaulay.  In fact, $C_4(1)$
is the smallest well-covered circulant graph with this property.
The graph $C_8(2)$ consists of two disjoint copies of $C_4(1)$:
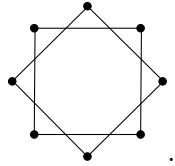
\[\Circulant{8}{2}.\]
By Lemma \ref{buchsbaumunion}, $C_8(2)$ cannot be Buchsbaum.  
In fact, $C_8(2)$ is the smallest well-covered circulant
that is not Buchsbaum.  An example of a connected well-covered 
circulant that is not Buchsbaum is
presented in Theorem \ref{minex}
\end{example}

The following lemma will
simplify some of our future arguments.

\begin{lemma}\label{circulantlemma}
Let $G = C_n(S)$ be a well-covered circulant graph.  
\begin{enumerate}
\item[$(i)$]
If 
${\rm link}_{{\rm Ind}(G)}(0)$ is Cohen-Macaulay, then $G$ is Buchsbaum.
\item[$(ii)$] If $\dim {\rm Ind}(G) = 1$, then $G$ is Buchsbaum.
\item[$(iii)$] If $\dim {\rm Ind}(G) = 2$, then $G$ is Buchsbaum
if ${\rm link}_{{\rm Ind}(G)}(0)$ is connected.
\end{enumerate}
\end{lemma}

\begin{proof}
$(i)$  By the symmetry of the graph, ${\rm link}_{{\rm Ind}(G)}(0) \cong 
{\rm link}_{{\rm Ind}(G)}(i)$ for all vertices $i \in \{1,\ldots,n-1\}$.   So,
to check if $G$ is Buchsbaum, it suffices to check  that ${\rm link}_{{\rm Ind}(G)}(0)$
is Cohen-Macaulay.  For $(ii)$, respectively $(iii)$, we use the
fact that $\dim {\rm link}_{{\rm Ind}(G)}(0) = 0$, respectively, $1$,
and then apply Theorem \ref{facts} $(ii)$, respectively, $(iii)$.
\end{proof}

%%%%%%%%%%%%%%%%%%%%%%%%%%%%%%%%%%%%%%%%%%%%%%%%%%%%%%%%%%%%%%%%%%%%%%%%%%%%%%%%

\section{Circulants of the form $C_n(d+1,d+2,\ldots,\lfloor \frac{n}{2} \rfloor)$}

In this section we determine the properties of the independence complex
of well-covered circulants of the form $C_n(d+1,d+2,\ldots,\lfloor \frac{n}{2} \rfloor)$
with $d \geq 1$.
These graphs are sometimes called the complement of the powers of cycles
because they are the complement of $C_n(1,2,\ldots,d)$.

Brown and Hoshino determined all the values of $n$ and $d$ such that
$G =C_n(d+1,d+2,\ldots,\lfloor \frac{n}{2} \rfloor)$ is well-covered,
i.e., ${\rm Ind}(G)$ is a pure simplicial complex:

\begin{theorem}[{\cite[Theorem 4.2]{Brown11}}]\label{wellcoveredcomplement}
Let $n$ and $d$ be integers with $n \geq 2d+2$ and $d \geq 1$.  Then
$C_n(d+1,d+2,\ldots,\lfloor \frac{n}{2} \rfloor)$ is well-covered if and only
if $n > 3d$ or $n = 2d+2$.
\end{theorem}

The 
$f$-vectors and $h$-vectors for ${\rm Ind}(C_n(d+1,d+2,\ldots,\lfloor \frac{n}{2} \rfloor))$
for some $n$ and $d$ are given below.

\begin{lemma}\label{hvectorcomplement}
Let $n$ and $d$ be integers with $n > 3d$ and $d \geq 1$.  If $G =
C_n(d+1,d+2,\ldots,\lfloor \frac{n}{2} \rfloor)$, then the 
$f$-vector of ${\rm Ind}(G)$ is given by
\[f({\rm Ind}(G)) = 
\left(1,n,\binom{d}{1}n,\binom{d}{2}n,\binom{d}{3}n,\ldots,\binom{d}{d-1}n,\binom{d}{d}n\right).\]
Consequently, $h({\rm Ind}(G))$, the $h$-vector of ${\rm Ind}(G)$, is
\[\left(1,n-(d+1),(-1)^2\binom{d+1}{2},(-1)^3\binom{d+1}{3},\ldots,(-1)^d\binom{d+1}{d},
(-1)^{d+1}\binom{d+1}{d+1} \right).\]

\end{lemma}

\begin{proof}
By \cite[Theorem 3.2]{Brown11}, the independence polynomial of $G = C_n(d+1,d+2,\ldots,\lfloor \frac{n}{2} \rfloor)$
when $n > 3d$ and $d \geq 1$ is given by
\[I=I\left(C_n\left(d+1,\ldots,\left\lfloor \frac{n}{2} \right\rfloor \right),x\right) = 1 + nx(1+x)^d.\]
The coefficient of $x^{i}$ in $I(G,x)$ counts the number of independent
sets of size $i$ in $G$.  So, the coefficient of $x^i$ is precisely $f_{i-1}$, the
number of faces of ${\rm Ind}(G)$ of dimension $(i-1)$.  Thus
$f({\rm Ind}(G))$ can now be computed by expanding out the polynomial $I$.

The $h$-vector of ${\rm Ind}(G)$ is computed from the $f$-vector using \eqref{hvector}.
We omit the details, but we note that the details can be found in \cite[Theorem 4.64]{HPhD}.
\end{proof}

The main result of this section
refines Theorem \ref{wellcoveredcomplement};  in particular,
all the pure independence  complexes of Theorem \ref{wellcoveredcomplement}
are
either vertex decomposable or Buchsbaum.

\begin{theorem}\label{thmcomplements}
Let $n$ and $d$ be integers with $n \geq 2d+2$ and $d \geq 1$.  
The following are equivalent
\begin{enumerate}
\item[$(i)$] $C_n(d+1,d+2,\ldots,\lfloor \frac{n}{2} \rfloor)$ is Buchsbaum.
\item[$(ii)$]  $C_n(d+1,d+2,\ldots,\lfloor \frac{n}{2} \rfloor)$ is well-covered.
\item[$(iii)$]  $n > 3d$ or $n = 2d+2$.
\end{enumerate}
Furthermore,
$C_n(d+1,d+2,\ldots,\lfloor \frac{n}{2} \rfloor)$ is vertex decomposable/shellable/Cohen-Macaulay
if and only if $n=2d+2$ and $d \geq 1$, or $d=1$ and $n > 3$.
\end{theorem}

\begin{proof}
Let $G = C_n(d+1,d+2,\ldots,\lfloor \frac{n}{2} \rfloor)$ with $n \geq 2d+2$ and $d \geq 1$.

The equivalence of $(ii)$ and $(iii)$ is simply Theorem \ref{wellcoveredcomplement}.   Furthermore,
if $G$ is Buchsbaum, then $G$ must be well-covered, so $(i)$ implies $(ii)$.  It suffices
to show that if $G$ is well-covered, it is also Buchsbaum.

If $n=2d+2$, then $G = C_{2d+2}(d+1)$, which implies that $G$ is $(d+1)$ disjoint copies of 
the complete graph $K_2$.
Since a $K_2$ is vertex decomposable, $G$ must be vertex decomposable (Lemma \ref{connected}), 
and hence Buchsbaum.

So, suppose that $n >3d$ and $d \geq 1$.  
We first note that ${\rm Ind}(G)$ has dimension $d$ from its $f$-vector
in Lemma \ref{hvectorcomplement}.  
Each element of the set
\[ \left\{ \{i,i+1,i+2,\ldots,i+d\} ~|~ \mbox{$0 \leq i \leq n-1$}\right\}\]
where the indices are computed modulo $n$, is an independent set in $G$.
Because $f_{d} = n$, and because each element of the above set is distinct,
these elements form a complete list of the facets of ${\rm Ind}(G)$.  

If $d=1$, then ${\rm Ind}(G)$ is connected, so it is vertex decomposable by
Theorem \ref{facts} and hence  Buchsbaum.  If $d > 1$, the facets
\[\{n-d,n-d+1,\ldots,0\},\{n-d+1,n-d+2,\ldots,0,1\},
\ldots,\{0,1,\ldots,d\}\]
are a complete list of the facets that contain $0$.  Thus, the facets of
${\rm link}_{{\rm Ind}(G)}(0)$ are
\[ \{n-d,n-d+1,\ldots,n-1\},\{n-d+1,n-d+2,\ldots,n-1,1\},\]
\[\{n-d+2,n-d+3,\ldots,n-1,1,2\},
\ldots,\{1,\ldots,d\}.\]
It follows from the order in which we have written these facets that 
${\rm link}_{{\rm Ind}(G)}(0)$ is shellable, and hence, by Theorem \ref{facts} Cohen-Macaulay.  
So, $G$ is Buchsbaum by Lemma \ref{circulantlemma}.

Observe that when $n=2d+2$, or $n > 3$ and $d=1$, then 
${\rm Ind}(G)$ is vertex decomposable, and so vertex decomposable, shellable, and 
Cohen-Macaulay.  On the other hand, if $n > 3d$ and $d \geq 2$,
then by Lemma \ref{hvectorcomplement},
the entry $h_3$ of $h({\rm Ind}(G))$ is negative.  Thus, by Theorem \ref{facts} $(v)$,
$G$ is not Cohen-Macaulay (and thus,
not vertex decomposable or shellable either).  This completes the proof of the final statement.
\end{proof}

\begin{remark}
Hoshino \cite[Theorem 4.64]{HPhD} first characterized when  $C_n(d+1,\ldots,\lfloor \frac{n}{2} \rfloor)$ 
is shellable.  Our theorem shows that the independence complex still has some structure
when the graph is not shellable, and moreover, it has a stronger structure if it is shellable.
\end{remark}

\begin{remark}
Hibi \cite{Hbook} is attributed with first asking for a 
characterization  of the $h$-vectors of Buchsbaum simplicial complexes.
This question remains open (see \cite{Mu,T} for some work on this problem).
However,  by the above theorem and Lemma \ref{hvectorcomplement}, 
\[\left(1,n-(d+1),\binom{d+1}{2},-\binom{d+1}{3},\ldots,(-1)^d\binom{d+1}{d},
(-1)^{d+1}\binom{d+1}{d+1} \right)\]
is a valid $h$-vector of a $d$-dimensional Buchsbaum simplicial complex on $n$ vertices
with $n > 3d$ and $d \geq 2$.
\end{remark}

%%%%%%%%%%%%%%%%%%%%%%%%%%%%%%%%%%%%%%%%%%%%%%%%%%%%%%%%%%%%%%%%%%%%%%%%%%%%%%%%%%

\section{Circulants of the form $C_n(1,\ldots,\hat{i},\ldots,\lfloor \frac{n}{2} \rfloor)$}

Moussi's thesis \cite{M} contains a number of families of well-covered
circulants.  We analyze the family $G = C_n(S)$ with $|S| = \lfloor \frac{n}{2} \rfloor - 1$.  
As shown in \cite{M}, all circulants in this family are well-covered (below, 
$\alpha(G)$ denotes the size of the largest independent set of $G$):  

\begin{theorem}[{\cite[Theorem 6.4]{M}}]\label{class11}
Let $G = C_n(S)$ be the circulant graph with $S = \{1,\ldots,\hat{i},\ldots,\lfloor \frac{n}{2} \rfloor\}$
for any $1 \leq i \leq \lfloor \frac{n}{2} \rfloor$.  Then $G$ is well-covered.  Furthermore,
$\alpha(G) = 2$ except if $i = \frac{n}{3}$, in which case, $\alpha(G) =3$.
\end{theorem}

As in the previous section, the well-covered circulants in this family can be divided into two groups, those
that are vertex decomposable and those that are merely Buchsbaum.

\begin{theorem}\label{theorem62}
Let $G = C_n(S)$ be the circulant graph with 
$S = \{1,\ldots,\hat{i},\ldots,\lfloor \frac{n}{2} \rfloor\}$
for any $1 \leq i \leq \lfloor \frac{n}{2} \rfloor$.  Then $G$ is Buchsbaum.  Furthermore
$G$ is vertex decomposable/shellable/Cohen-Macaulay if and only if ${\rm gcd}(i,n) = 1$.
\end{theorem}

\begin{proof}
Let $G = C_n(1,\ldots,\hat{i},\ldots,\lfloor \frac{n}{2} \rfloor)$ for some 
$i \in \{1,\ldots, \lfloor \frac{n}{2} \rfloor\}$.  By Theorem \ref{class11}, ${\rm Ind}(G)$ is pure.  

If $i \neq \frac{n}{3}$, then $\dim {\rm Ind}(G) =  1$ since $\alpha(G) =2$ by Theorem \ref{class11}.  Thus,
$G$ is Buchsbaum by Lemma \ref{circulantlemma} $(ii)$.
If $i = \frac{n}{3}$, the $\dim {\rm Ind}(G) =2$.
In particular,
\begin{equation}\label{facetsremoveone}
{\rm Ind}(G) = \langle \{0, i, 2i\}, \{1, i+1, 2i+1\},\ldots,\{i-1, 2i-1, 3i-1\}\rangle.
\end{equation}
Because ${\rm link}_{{\rm Ind}(G)}(0) = \langle \{i,2i\} \rangle$ is connected,
we can apply Lemma \ref{circulantlemma}.  So, $G$ is always Buchsbaum.

We now prove the second statement.  We treat the cases $i \neq \frac{n}{3}$ and 
$i = \frac{n}{3}$ separately.  

If  $i \neq \frac{n}{3}$, then $\dim {\rm Ind}(G) =  1$.
So, by Theorem \ref{facts}, it suffices to show that ${\rm Ind}(G)$ is connected if and only
${\rm gcd}(i,n) =1$.

If ${\rm gcd}(i,n) = 1$, then the map $\phi:\mathbb{Z}_n \rightarrow \mathbb{Z}_n$ given by
$\phi(j) = ji$ is a bijection.  So, the elements $\{0,i,2i,3i,\ldots,(n-1)i\}$ are all distinct
elements.  But this implies that  
\[\{0,i\},\{i,2i\},\{2i,3i\},\ldots,\{(n-1)i,0\}\]
is path of facets in ${\rm Ind}(G)$ that includes all the vertices, and hence
${\rm Ind}(G)$ is connected.

On the other hand, suppose that ${\rm gcd}(i,n) = k > 1$.  
Then $\{0,i,2i,\ldots,(n/k-1)i\}$ and 
$\{1,i+1,\ldots,(n/k-1)i+1\}$ are disjoint sets in $\mathbb{Z}_n$.  
Now every vertex $a$ is non-adjacent in $G$ to exactly two vertices, namely $a+i$ and $a-i$ (modulo $n$).  
But then there is no path from the connected facets
\[\{0,i\},\{i,2i\},\ldots,\{(n/k-1)i,0\}\]
to any of the connected facets
\[\{1,i+1\},\{i+1,2i+1\},\ldots,\{(n/k-1)i+1,1\}.\]
In other words, ${\rm Ind}(G)$ is disconnected. 

If $i = \frac{n}{3}$, then ${\rm gcd}(i,n) =1$ if and only if $i=1$, i.e., $n =3$.  
The facets of ${\rm Ind}(G)$ are given in \eqref{facetsremoveone}.
If $i=1$ and $n=3$, then ${\rm Ind}(G)$ is simply the simplex with unique maximal facet 
$\{ 0,1,2\}$,
and thus, it is vertex decomposable.  If $i > 1$, then ${\rm Ind}(G)$ is a disconnected simplicial 
complex of dimension two, so by Theorem \ref{facts} $(iv)$, 
it is not Cohen-Macaulay, and thus, not vertex-decomposable.
\end{proof}

%%%%%%%%%%%%%%%%%%%%%%%%%%%%%%%%%%%%%%%%%%%%%%%%%%%%%%%%%%%%%%%%%%%%%%%%%%

\section{One-paired circulants}\label{cnab} 

In \cite{BGM}, Boros {\it et al.} studied circulant graphs for which every maximal clique 
(a clique is a subgraph in which every vertex is adjacent to every other vertex)
intersects each maximal independent set.  Such a graph is called a CIS graph:

\begin{definition}
A graph $G$ is {\it CIS} if for every maximal clique $C$ and every maximal independent set 
$I$ in $G$, $C\cap I \neq \emptyset$. (CIS is an acronym for Cliques Intersect Stable sets).
\end{definition}

Boros {\it et al.} \cite[Theorem 3]{BGM} showed that CIS circulants are well-covered.
In fact, the main theorem of \cite{BGM} is a classification of circulant graphs
that are CIS; precisely, the circulant graph $G$ is CIS if and only if all maximal
independent sets have size $\alpha(G)$ and all maximal cliques have size $\omega(G)$, and
$\alpha(G)\omega(G) = |V(G)|$.

In \cite{BGM}, the authors also describe how to construct some CIS circulants graphs. One construction
is the one-paired circulants described below. We provide a more direct proof that
one-paired circulants are CIS (Corollary~\ref{omega}) by first characterizing their
structure (Theorem~\ref{structure}).  We then consider 
the independence complex of a one-paired circulant.

\begin{definition} 
The circulant graph $G=C_n(S)$ is \emph{one-paired} if there exists
an ordered pair of positive integers $(a,b)$
such that $ab|n$ 
and $S = \left\{ d\in \{1,\ldots,\lfloor \frac{n}{2} \rfloor\} ~: ~
\mbox{$a|d$ and $ab\nmid d$} \right\}$. 
We denote the one-paired circulant by $G = C(n;a,b)$.
\end{definition}

\begin{example}
Let $n=12$ and consider the ordered pair $(2,3)$.  Then $C(12;2,3)$ is the 
circulant graph $C_{12}(2,4)$.  Compare this graph to $C(12;3,2)$ which is $C_{12}(3)$.
\end{example}

We begin with a structural result for $C(n;a,b)$.   Note that
$G\vee H$, the \emph{join} of $G$ and $H$, is the graph with
vertex set $V_G \cup V_H$ and edges 
$E_G \cup E_H \cup \{\{x,y\} ~|~ x \in V_G ~\mbox{and}~ y \in V_H\}$.

\begin{theorem}\label{structure}
Let $G = C(n;a,b)$ be a one-paired circulant. Then
\[C(n;a,b)  = \bigcup_{i=1}^a \left(\bigvee_{j=1}^{b} \overline{K_{\frac{n}{ab}}} \right)
~~\mbox{and $\alpha(G) = \frac{n}{b}$}.\]
\end{theorem}

\begin{proof}
Let $G = C(n;a,b)$ be a one-paired circulant with $\alpha = \alpha(G)$. By the definition
of a one-paired circulant, $n=kab$ for some positive integer $k$. 
Let $I=\{0,ab,2ab,\ldots,(k-1)ab\}$. Note that $I$ is an independent set of size $k$ 
in $G$. By vertex transitivity of $G$,  the cosets in 
$W=\{I, I+a,I+2a,\ldots, I+(b-1)a\}$ are $b$ disjoint independent sets of size $k$. 
We claim that the subgraph of $G$ induced by $W$ is $\bigvee^b_{j=1} \overline{K_k}$. 
In particular, suppose $m_i\in I+ia$ and $m_r\in I+ra$ for
some  $i,r \in\{0,1,\dots, (b-1)\}$, $i>r$. Then $a| (m_i-m_r)$ but
$ab\nmid (m_i-m_r)$ since $(i-r)<b$. It follows that $m_i$ is adjacent to $m_r$ in
$G$. Hence the claim is established. 

The cosets $W,W+1,\ldots,W+(a-1)$ form a partition of the vertex set 
of $G$ into disjoint graphs on $kb$ vertices. Thus
$G=\bigcup_{i=1}^a {\left(\bigvee_{j=1}^{b}  \overline{K_{k}}\right)}$. 
It follows that $\alpha(G)=ak = n/b$.
\end{proof}

Setting $a=1$ in Theorem~\ref{structure} gives us the following corollary:

\begin{corollary}\label{substructure}
\[ C(n;1,b)= \bigvee_{j=1}^b \overline{K_{\frac{n}{b}}}.\]
\end{corollary}

Re-combining with 
Theorem~\ref{structure}, we can re-characterize the structure of one-paired circulants: 

\begin{corollary}\label{structure2}
\[C(n;a,b)  =  \bigcup_{i=1}^a C\left(\frac{n}{a};1,b\right).\]
\end{corollary}

As a result of Corollary~\ref{structure2}, when exploring one-paired circulants, 
we can focus on the the graphs $C(n;1,b)$. 
We can determine the clique number of a one-paired circulant.  In 
fact, we have a more direct proof of \cite[Theorem 4]{BGM} that 
every one-paired circulant is CIS.

\begin{corollary}\label{omega}
Let $G = C(n; a, b)$ be a one-paired circulant with $\alpha=\alpha(G)$.
Then $\omega(G) = b$, $n=\alpha b$, and $G$ is a CIS graph.
Furthermore, $G$ is well-covered.
\end{corollary}

\begin{proof}
By Corollary~\ref{structure2}, a maximum clique in $G$ must be
a maximum clique in $C(\frac{n}{a};1,b)$. By Corollary~\ref{substructure}, it
follows that $\omega(G)=b$. Thus by Theorem~\ref{structure}, $n=\alpha \omega$. It also follows from the structure
of $G$ that each maximal independent set intersects each clique of size $\omega$,
so $G$ is a CIS graph.
In addition, these independent sets all have the same size, i.e., $G$ is well-covered.
\end{proof}

We can also determine the $f$-vector of the independence complex 
${\rm Ind}(C(n;1,b))$ (equivalently,
the independence polynomial of $C(n;1,b)$) directly 
from the structural description.

\begin{theorem}\label{fcirc}
Let $G = C(n;1,b) = C(mb,1,b)$.  Then the $f$-vector of ${\rm Ind}(G)$ is 
\[f({\rm Ind}(G)) = \left(1,\binom{m}{1}b,\binom{m}{2}b,\ldots,\binom{m}{m-1}b,b\right).\]
\end{theorem}

\begin{proof}
This will be a counting argument based upon our description of the graph $C(n;1,b)$.
Suppose $G = C(mb,1,b)$. By Corollary~\ref{substructure},
$G = \overline{\bigcup_{j=1}^b K_m}$.
(Note that $\overline{\bigcup_{j=1}^b K_m} = \bigvee_{j=1}^b \overline{K_m}$.)
Hence, any non-adjacency is found within a $\overline{K_m}$.
Thus, for each independent set of size $k$, there will be $\binom{m}{k}$ 
choices of vertices in each of the $b$ copies of $\overline{K_m}$.
Hence, there are $\binom{m}{k}b$ independence sets of size $k$.
Thus, $f({\rm Ind}(G)) = (1,\binom{m}{1}b,\binom{m}{2}b,\ldots,\binom{m}{m-1}b,b)$.
\end{proof}

\begin{theorem}\label{one-pairedspecial}
Let $G$ be the one-paired circulant $G = C(mb;1,b)$.  Then $G$ is Buchsbaum.  
Furthermore, $G$ is vertex decomposable/shellable/Cohen-Macaulay if and only if $m=1$.
\end{theorem}

\begin{proof}
Suppose $G = C(mb,1,b)$.  By Theorem \ref{fcirc}, $\dim {\rm Ind}(G) = m-1$
with $b$ facets.  So, 
\[\{\{0+j,b+j,2b+j,\ldots,(m-1)b+j \} ~|~ j=0,\ldots,b-1\}\]
is a complete list of the facets of ${\rm Ind}(G)$, where addition is modulo $n$.

Since $0$ only appears in the facet $\{0,b,2b,\ldots,(m-1)b\}$, 
${\rm link}_{{\rm Ind}(G)}(0)$ is the simplex $\langle \{b,2b,\ldots,(m-1)b\} \rangle$,
and thus the link is vertex decomposable. So $G$ is Buchsbaum
by Lemma \ref{circulantlemma}.

If $m > 1$, then the facets of ${\rm Ind}(G)$ are disjoint, and thus
${\rm Ind}(G)$ is not connected.  As a consequence,
$G$ is not vertex decomposable, shellable, or Cohen-Macaulay
by Theorem \ref{facts} $(iii)$ and $(iv)$.  However, If $m=1$,
then ${\rm Ind}(G)$ has dimension $0$, and so is vertex decomposable, shellable,
and Cohen-Macaulay by Theorem \ref{facts} $(ii)$.
\end{proof}

\begin{corollary}
Let $G$ be the one-paired circulant $G = C(n;a,b)$. 
\begin{enumerate}
\item[$(i)$] $G$ is vertex decomposable/shellable/Cohen-Macaulay
if and only if $n=ab$.  
\item[$(ii)$] $G$ is Buchsbaum but not Cohen-Macaulay if and only if $a=1$ and $ab <n$.
\item[$(iii)$] ${\rm Ind}(G)$ is pure but not Buchsbaum if and only if $1 < a$ and $ab <n$. 
\end{enumerate}
\end{corollary}

\begin{proof}
$(i)$ By Corollary \ref{structure2}, $C(n;a,b)$ is the disjoint union of $a$ copies of 
$C(\frac{n}{a};1,b)$.
By Lemma \ref{connected} $C(n;a,b)$ will be vertex decomposable, shellable, or Cohen-Macaulay
if and only if each connected component  $C\left(\frac{n}{a};1,b\right)$ has this property.
But by Theorem \ref{one-pairedspecial},  $C\left(\frac{n}{a};1,b\right)$ has these properties 
if and only if $\frac{n}{a} = b$, i.e., $n = ab$. 

We now prove $(ii)$ and $(iii)$.  Note that ${\rm Ind}(G)$ is pure by Corollary \ref{omega}, so it suffices
to show that ${\rm Ind}(G)$ is Buchsbaum, but not Cohen-Macaulay, if and only if $a =1$.
If $a = 1$, then $G = C(n;a,b) = C(n;1,b)$.  By Theorem \ref{one-pairedspecial}, $G$ is Buchsbaum.
In addition, if $b <n$, $G$ cannot be Cohen-Macaulay.   

Now suppose that $a \geq 2$ and $ab < n$. By Theorem \ref{one-pairedspecial} 
and Corollary \ref{structure2}, $C(n;a,b)$ is the disjoint union of $a \geq 2$
Buchsbaum graphs that are not Cohen-Macaulay.  
Then by Lemma \ref{buchsbaumunion}, $C(n;a,b)$ cannot be Buchsbaum.
\end{proof}

\begin{example}
The graph $C_8(2)$ equals the one-paired circulant $C(8;2,2)$.  Since $2\cdot 2 < 8$ and $1<2$,
we have that $C_8(2)$ is pure but not Buchsbaum by the above result.  This was already noted in
Example \ref{example1}.
\end{example}

%%%%%%%%%%%%%%%%%%%%%%%%%%%%%%%%%%%%%%%%%%%%%%%%%%%%%%%%%%%%%%%%%%%%%%%%%%%%%
\section{Minimal examples}

Recall that the following implications always hold for 
pure simplicial complexes:
\[\mbox{vertex decomposable} \Rightarrow \mbox{shellable} \Rightarrow
\mbox{Cohen-Macaulay} \Rightarrow \mbox{Buchsbaum}. \]
For many families of well-covered graphs, the reverse implications also hold,
for example, see \cite{DE,HHZ,W} for chordal graphs and \cite{EV,VT} for bipartite graphs.
However, if we restrict to well-covered circulants, the reverse implications
may fail to hold.  

Example \ref{example1} already gives an example of a well-covered circulant that
is not Buchsbaum.  As we detail in the next section, we used computer algebra systems
to determine all well-covered circulants on $n\leq 16$ vertices (see
Table \ref{table}).  
Our computer search found the following
vertex minimal counter-examples to the reverse implications.

\begin{theorem}\label{minex}
\begin{enumerate}
\item[$(i)$] The disconnected graph $C_8(2)$ is the
smallest well-covered circulant that is not Buchsbaum.
The well-covered circulant $C_{10}(1,4)$ is the smallest connected 
well-covered circulant that is not Buchsbaum.
\item[$(ii)$] The graph $C_4(1)$ is the smallest well-covered circulant 
that is Buchsbaum but not Cohen-Macaulay.
\item[$(iii)$] The graph $C_{16}(1,4,8)$ is the smallest well-covered circulant 
that is shellable but not vertex decomposable.
\end{enumerate}
\end{theorem}

\begin{proof}
The minimality is a result of our computer search.  Example \ref{example1} showed that $C_4(1)$
is Buchsbaum but not Cohen-Macaulay and $C_8(2)$ is not Buchsbaum.  There are only two facets
of $\Delta = {\rm Ind}(C_{10}(1,4))$ that contain $0$, namely $\{0,3,5,8\}$ and $\{0,2,5,7\}$.  So
${\rm link}_{\Delta}(0) = \langle \{3,5,8\},\{2,5,7\} \rangle$.  But then the link
has $f$-vector $(1,5,6,2)$ and $h$-vector $(1,2,-1,0)$, so the link is not Cohen-Macaulay,
and thus, $C_{10}(1,4)$ is not Buchsbaum.

For $G = C_{16}(1,4,8)$ (see Figure \ref{c16148}), 
\begin{figure}[h]
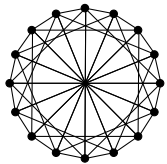

\[\Circulant{16}{1,4,8}\]
\caption{The circulant graph $C_{16}(1,4,8)$.}
\label{c16148}
\end{figure}
the deletion ${\rm del}_{{\rm Ind}(G)}(0)$ has $f$-vector 
$(1,15,70,117,60)$, and thus,
$h$-vector $(1, 11, 31, 18, -1)$.  So, by Theorem \ref{facts} $(v)$, the deletion is not
Cohen-Macaulay, so the deletion cannot be vertex decomposable.  Because of the 
symmetry of the graph, $G$ has no vertex $x$ such that ${\rm del}_{{\rm Ind}(G)}(x)$ is vertex decomposable,
and so $G$ cannot be vertex decomposable.  The simplicial complex ${\rm Ind}(G)$ has $80$
facets.  For completeness, here is one shelling order read left to right (the independence complex has dimension
three, so each group of four integers represents a facet):
\scriptsize
\begin{center}
\begin{verbatim}
8 10 13 15   4 10 13 15   6 8 13 15   2 8 13 15   4 6 13 15   2 4 13 15   5 8 10 15   
5 10 12 15   1 10 12 15   6 9 12 15   1 6 12 15   1 8 10 15   1 4 10 15   4 6 9 15   
2 4 9 15     2 9 12 15    2 5 12 15   1 6 8 15    2 5 8 15    1 4 6 15    6 8 11 13   
4 9 11 14    7 9 12 14    3 9 12 14   0 9 11 14   4 7 9 14    0 7 9 14    0 3 9 14   
2 8 11 13    2 7 9 12     2 4 7 13    4 7 10 13   5 7 10 12   5 7 12 14   3 5 12 14 
0 5 7 14     0 5 11 14    5 8 11 14   3 5 8 14    0 3 5 14    1 7 10 12   1 7 12 14   
1 3 12 14    1 8 11 14    1 3 8 14    3 5 10 12   3 8 10 13   3 6 8 13    1 3 10 12   
2 5 7 12     0 6 9 11     0 6 11 13   0 2 11 13   0 7 10 13   0 2 7 13    4 6 9 11   
4 6 11 13    2 4 11 13    2 4 9 11    0 2 9 11    1 6 8 11    2 5 8 11    1 4 6 11   
1 4 11 14    1 4 7 14     0 2 5 11    3 5 8 10    1 3 8 10    0 5 7 10    1 4 7 10   
0 3 5 10     0 3 10 13    0 3 6 13    2 4 7 9     0 2 7 9     0 3 6 9     3 6 9 12 
1 3 6 12     1 3 6 8      0 2 5 7
\end{verbatim}
\end{center}
\normalsize
\end{proof}

\begin{remark}\label{vdnotshellable}
To the best of our knowledge, $C_{16}(1,4,8)$ is the first example
of an independence complex that is shellable but not vertex decomposable.
Provan and Billera's original paper on vertex decomposability (see \cite{PB}) points
out that Walkup's example (see \cite{W1978}) of a simplicial complex on 56 vertices
and over 8000 27-simplicies is an example of a shellable but not vertex decomposable simplicial
complex.  (Provan and Billera proved that any vertex decomposable simplicial complex satisfies
Hirsh's conjecture, while Walkup's example was a counter-example to this example.)
The short note of  Moriyama and  Takeuchi \cite{MT} contains a list of the minimal
two dimensional simplicial complexes that are vertex decomposable, but not shellable.  We
checked these complexes, but none of them are the independence complex of a graph.
Our example is also interesting for the following reason.  It is known from \cite{PB}
that the barycentric subdivision of any shellable simplicial complex is vertex decomposable.
So, the independence complex of $C_{16}(1,4,8)$ cannot be constructed by taking the
barycentric subdivision of a shellable simplicial complex.\footnote[1]{After submitting this
paper, we observed that $C_{16}(1,4,8)$ also gives a negative answer to
\cite[Conjecture 2]{V}.  R. Villarreal conjectured that every Cohen-Macaulay
graph $G$ has a vertex $x$ such that $G \setminus x$ is also Cohen-Macaulay.  It
was already known that this conjecture is false due to an example of Terai
\cite[Exercise 6.2.24]{Vbook}.  However
Terai's example does not hold for all characteristics.  Our example works
in all characteristics.  To see why, note that $C_{16}(1,4,8)$ is shellable, so it is
Cohen-Macaulay over any field.  As we observed in the proof of
Theorem \ref{minex}, for
any vertex $x$ of $C_{16}(1,4,8)$, the $h$-vector of the independence complex of 
$C_{16}(1,4,8) \setminus x$, which only depends upon the combinatorics of the complex,
implies that $C_{16}(1,4,8) \setminus x$ is not Cohen-Macaulay.
}

\end{remark}

\begin{remark}
The {\it lexicographical product} of the graphs $G$ and $H$,
denoted $G[H]$, is the graph with the vertex set $V_G \times V_H$ and
where $(w,x)$ and $(y,z)$ are adjacent if $\{w,y\} \in E_G$ or if $w=y$, then
$\{x,z\} \in E_H$.  In a forthcoming paper, we will explore how the topological properties
(e.g., vertex decomposable, shellable)
of ${\rm Ind}(G)$ and ${\rm Ind}(H)$ are preserved under the lexicographical product.  As an application
we will use the graph $C_{16}(1,4,8)$ to build an infinite family of graphs
which are shellable but not vertex decomposable.
\end{remark}

The reader will notice that there is no example of a 
well-covered circulant that is Cohen-Macaulay
but not shellable.    We know of no such a graph, so we leave it as a question:

\begin{question}
Is there a well-covered circulant that is Cohen-Macaulay but not shellable?
\end{question}

In fact, we are not aware of any graph whose independence complex is Cohen-Macaulay but
not shellable.

%%%%%%%%%%%%%%%%%%%%%%%%%%%%%%%%%%%%%%%%%%%%%%%%%%%%%%%%%%%%%%%%%%%%%%%%

\section{Additional Observations and Computations}

Using {\it Macaulay2} \cite{Mt} and Sage \cite{Sage}
we determined all well-covered circulants on $n \leq 16$ vertices,
and determined their combinatorial topological properties.  This information
is collated in Table \ref{table} at the end of the section.  We record a number of observations about the 
independence complex of well-covered circulants, some of which are based upon
our computer experiments.

\subsection{$1$-well-covered circulants}

A well-covered graph $G$ is said to be {\it $1$-well-covered} if $G \setminus \{x\}$,
the graph with the vertex $x$ and of its adjacent edges removed,
is a well-covered graph for all vertices $x\in V$.  This notion
was introduced by Staples \cite{S}.  In \cite[Theorem 3.3]{M}, Moussi
determined which of the well-covered circulants of the form
$G=C_n(1,2,\ldots,d)$ were also $1$-well-covered.  

What is
striking about \cite[Theorem 3.3]{M} is that the class of $1$-well-covered circulants
of the form $G = C_n(1,2,\ldots,d)$ coincides exactly with those 
that are vertex decomposable, as first found in \cite[Theorem 3.4]{VVW}.
This observation suggests some connection between the two concepts.  Indeed, 
vertex decomposability implies $1$-well-coveredness:

\begin{theorem}\label{vdwellcovered}
Let $G$ be a circulant graph.  If $G$ is vertex decomposable,
then $G$ is $1$-well-covered.
\end{theorem}

\begin{proof}
We begin with the straight-forward observation that
for any graph $G$, we have  ${\rm del}_{{\rm Ind}(G)}(x) = {\rm Ind}(G\setminus \{x\})$
for any vertex $x$ of $G$.  
Now, if $G$ is vertex decomposable, then there exists a
vertex $x$ such that ${\rm del}_{{\rm Ind}(G)}(x)$ is vertex decomposable, hence pure.
But by symmetry, all vertices $x$ will have this property.  
By the above observation,  this means that ${\rm Ind}(G\setminus x)$ is
pure, i.e., $G\setminus x$ is well-covered for all $x$.
\end{proof}

The above result allows us to give a new proof for \cite[Theorem 3.3]{M}.

\begin{theorem} 
Let $n$ and $d$ be integers with $n\geq 2d\geq 2$ and let $G=C_n(1,2,\ldots,d)$.
Then  $G$ is $1$-well-covered if and only if $n \leq 3d+2$ and $n \neq 2d+2$.
\end{theorem}
\begin{proof}
By \cite[Theorem 4.1]{Brown11}, $G$ is well-covered if and only if $n\leq 3d+2$ or $n =4d+3$.
Since a $1$-well-covered graph must also be well-covered, we only need to look at the
cases $n \leq 3d+2$ or $n=4d+3$.
If $n \leq 3d+2$ and $n \neq 2d+2$, then $G$ is vertex decomposable by
\cite[Theorem 3.4]{VVW}, so by
Theorem \ref{vdwellcovered}, $G$ is $1$-well-covered.  

It suffices to show if $n=2d+2$ or $n=4d+3$, $G$ is not $1$-well-covered.    
If $n =2d+2$, consider the graph $G\setminus \{0\}$.  Then the vertex $d+1$ is adjacent
to every other vertex, so $\{d+1\}$ is a maximal independent set.  However, $\{1,d+2\}$
is also an independent set, so $G \setminus \{0\}$ is not well-covered.
When $n=4d+3$,  we again consider the graph $G \setminus \{0\}$
which is a graph on the vertices $\{1,2,\ldots,4d+2\}$.  The set
$\{d+1,3d+2\}$ is a maximal independent set in this graph.  To
see this, note that $d+1$ is adjacent to $\{1,\ldots,d,d+2,\ldots,2d+1\}$,
and $3d+2$ is adjacent to $\{2d+2,\ldots,3d-1,3d+1,\dots 4d+2\}$.  On the 
other hand, $\{1,d+2,3d+3\}$ is an independent set of size $3$, so $G \setminus \{0\}$
is not well-covered.
\end{proof}

\begin{example}
When we computed our table of well-covered circulants on $n \leq 16$ vertices,
we also checked which of these circulants were $1$-well-covered (see Table
\ref{table}).
In particular, by Table \ref{table}, the converse of Theorem \ref{vdwellcovered} is false. 
The graph $C_{10}(1,2,3,5)$ is Buchsbaum but not vertex decomposable.  However,
it is still $1$-well-covered.  Furthermore, this is the minimal such example with respect
to the number of vertices.  Note that the fact that $C_{10}(1,2,3,5)$ is Buchsbaum but not
vertex decomposable can also be deduced from 
Theorem \ref{theorem62} since ${\rm gcd}(4,10) \neq 1$.
\end{example}

We pose the following question based upon on our computations:

\begin{question}
Can the hypotheses of Theorem \ref{vdwellcovered} be relaxed, i.e., does
the conclusion still hold if $G$ is shellable or Cohen-Macaulay?
\end{question}

One of the difficulties in answering this question is that in most
cases in which we find a shellable or Cohen-Macaulay circulant graph, it
is also vertex decomposable.  Our new example of a graph that is
shellable but not vertex decomposable, i.e., the graph $C_{16}(1,4,8)$, 
is also $1$-well-covered.

%%%%%%%%%%%%%%%%%%%%%%%%%%%%%%%%%%%%%%%%%%%%%%%%%%%%%%%%%%%%%%%%%%%%%%%%%%%

\subsection{Circulant Cubic Graphs}

A {\it cubic} graph is a graph such that every vertex has degree three. Brown 
and Hoshino \cite[Theorem 4.3]{Brown11} determined which connected cubic 
circulant graphs were well-covered.  The last two authors and Watt refined 
this result to determine which of these graphs were Cohen-Macaulay (see
\cite[Theorem 5.2]{VVW}), while Hoshino \cite[Proposition 4.61]{HPhD}
determined which ones were shellable.  As a consequence of our computations, we observed 
that the Cohen-Macaulay connected cubic circulants were also vertex 
decomposable.  More precisely, we have the following result, which 
summarizes the past theorems and our computations.

\begin{theorem}
Let $G$ be a connected cubic circulant graph.  Then $G$ is well-covered if and only it is
isomorphic to $C_4(1,2)$, $C_6(1,3)$, $C_6(2,3)$, $C_8(1,4)$, or $C_{10}(2,5)$.  In 
addition
\begin{enumerate}
\item[$(i)$] $C_4(1,2)$ and $C_6(2,3)$ are vertex decomposable.
\item[$(ii)$] $C_6(1,3)$,$C_8(1,4)$, $C_{10}(2,5)$ are Buchsbaum but not 
Cohen-Macaulay.
\end{enumerate}
\end{theorem}
  
%%%%%%%%%%%%%%%%%%%%%%%%%%%%%%%%%%%%%%%%%%%%%%%%%%%%%%%%%%%%%%%%%%%%%%%%%

\subsection{Well-covered circulants of small order}

By a computer search,
we have determined all well-covered circulant graphs on $3 \leq n \leq 16$ vertices.
For each well-covered graph, we determined if it was vertex decomposable,
shellable, Cohen-Macaulay, and/or Buchsbaum.  
We also determined if the circulant was connected
or not, and whether or not it was $1$-well-covered. 
Our computations were made using Sage \cite{Sage} and {\it Macaulay2} \cite{Mt}.  The
{\it Macaulay2} packages {\tt EdgeIdeals} \cite{FHVT} and {\tt SimplicialDecomposability}
\cite{Cook} were also used to carry out our experiments.

Table \ref{table} contains the following information.  Every circulant $G = C_n(a_1,\ldots,a_t)$
in the table is well-covered.  
If there is a circulant $C_n(b_1,\ldots,b_t) \cong 
C_n(a_1,\ldots,a_t)$, we only list one circulant.
A $*$ is used to indicated 
that the graph is disconnected.  
Because we have the implications,
\[\mbox{vertex decomposable} \Rightarrow \mbox{shellable} \Rightarrow
\mbox{Cohen-Macaulay} \Rightarrow \mbox{Buchsbaum} \]
it is enough to know the strongest structure $C_n(a_1,\ldots,a_t)$ posses.
We therefore write $V$ if $G$ is vertex decomposable, $S$ if $G$ is shellable but not vertex decomposable, 
$B$ if $G$ is Buchsbaum
but not Cohen-Macaulay, and $N$ if $G$ has none of these properties.  Finally, if
$G$ is $1$-well-covered, we denote this in a separate column by a $1$.

\newpage

\begin{center}
\begin{tiny}
\captionof{table}{Well-Covered circulant graphs up to order 16. 
}\label{table}
\begin{tabular}{||l| c c|| l| c c|| l| c c|| l| c c|}
\hline
\hline
$C_{3}(1)$ & V & 1 & $C_{11}(1, 2, 3, 4, 5)$ & V & 1 & $C_{14}(2, 4, 6)*$ & V & 1 & $C_{16}(2, 4)*$ & V & 1\\
\cline{0-5}
$C_{4}(1)$ & B &  & $C_{12}(3)*$ & N &  & $C_{14}(2, 4, 7)$ & B &  & $C_{16}(2, 6)*$ & N & \\
$C_{4}(2)*$ & V & 1 & $C_{12}(4)*$ & V & 1 & $C_{14}(1, 2, 3, 4)$ & V & 1 & $C_{16}(2, 8)*$ & N & \\
$C_{4}(1, 2)$ & V & 1 & $C_{12}(6)*$ & V & 1 & $C_{14}(1, 2, 3, 7)$ & B &  & $C_{16}(4, 8)*$ & V & 1\\
\cline{0-2}
$C_{5}(1)$ & V & 1 & $C_{12}(1, 4)$ & N &  & $C_{14}(1, 2, 4, 6)$ & V & 1 & $C_{16}(1, 2, 4)$ & B & \\
$C_{5}(1, 2)$ & V & 1 & $C_{12}(2, 4)*$ & N &  & $C_{14}(1, 2, 4, 7)$ & B &  & $C_{16}(1, 2, 6)$ & B & 1\\
\cline{0-2}
$C_{6}(2)*$ & V & 1 & $C_{12}(2, 6)*$ & N &  & $C_{14}(1, 2, 5, 6)$ & N &  & $C_{16}(1, 4, 6)$ & V & 1\\
$C_{6}(3)*$ & V & 1 & $C_{12}(3, 4)$ & B &  & $C_{14}(1, 2, 5, 7)$ & B & 1 & $C_{16}(1, 4, 7)$ & N & \\
$C_{6}(1, 2)$ & B &  & $C_{12}(3, 6)*$ & V & 1 & $C_{14}(1, 3, 5, 7)$ & B &  & $C_{16}(1, 4, 8)$ & S & 1\\
$C_{6}(1, 3)$ & B &  & $C_{12}(4, 6)*$ & V & 1 & $C_{14}(1, 4, 6, 7)$ & B & 1 & $C_{16}(1, 6, 8)$ & B & \\
$C_{6}(2, 3)$ & V & 1 & $C_{12}(1, 2, 6)$ & B & 1 & $C_{14}(2, 4, 6, 7)$ & V & 1 & $C_{16}(2, 4, 6)*$ & N & \\
$C_{6}(1, 2, 3)$ & V & 1 & $C_{12}(1, 3, 5)$ & B &  & $C_{14}(1, 2, 3, 4, 5)$ & V & 1 & $C_{16}(2, 4, 8)*$ & V & 1\\
\cline{0-2}
$C_{7}(1)$ & B &  & $C_{12}(1, 3, 6)$ & V & 1 & $C_{14}(1, 2, 3, 4, 6)$ & V & 1 & $C_{16}(2, 6, 8)*$ & N & 1\\
$C_{7}(1, 2)$ & V & 1 & $C_{12}(1, 4, 6)$ & B &  & $C_{14}(1, 2, 3, 4, 7)$ & V & 1 & $C_{16}(1, 2, 3, 8)$ & B & 1\\
$C_{7}(1, 2, 3)$ & V & 1 & $C_{12}(2, 3, 4)$ & N &  & $C_{14}(1, 2, 3, 5, 7)$ & B &  & $C_{16}(1, 2, 4, 7)$ & B & \\
\cline{0-2}
$C_{8}(2)*$ & N &  & $C_{12}(2, 3, 6)$ & B & 1 & $C_{14}(1, 2, 3, 6, 7)$ & B &  & $C_{16}(1, 2, 5, 8)$ & B & \\
$C_{8}(4)*$ & V & 1 & $C_{12}(2, 4, 6)*$ & V & 1 & $C_{14}(1, 2, 4, 6, 7)$ & V & 1 & $C_{16}(1, 2, 6, 7)$ & N & \\
$C_{8}(1, 2)$ & V & 1 & $C_{12}(3, 4, 6)$ & B & 1 & $C_{14}(1, 2, 5, 6, 7)$ & V & 1 & $C_{16}(1, 2, 6, 8)$ & N & \\
$C_{8}(1, 3)$ & B &  & $C_{12}(1, 2, 3, 4)$ & V & 1 & $C_{14}(1, 2, 3, 4, 5, 6)$ & B &  & $C_{16}(1, 2, 7, 8)$ & B & \\
$C_{8}(1, 4)$ & B &  & $C_{12}(1, 2, 4, 5)$ & B &  & $C_{14}(1, 2, 3, 4, 5, 7)$ & B & 1 & $C_{16}(1, 3, 5, 7)$ & B & \\
$C_{8}(2, 4)*$ & V & 1 & $C_{12}(1, 2, 4, 6)$ & V & 1 & $C_{14}(1, 2, 3, 4, 6, 7)$ & V & 1 & $C_{16}(1, 4, 6, 8)$ & B & \\
$C_{8}(1, 2, 3)$ & B &  & $C_{12}(1, 3, 4, 5)$ & B & 1 & $C_{14}(1, 2, 3, 4, 5, 6, 7)$ & V & 1 & $C_{16}(1, 4, 7, 8)$ & B & 1\\
\cline{7-9}
$C_{8}(1, 2, 4)$ & V & 1 & $C_{12}(1, 3, 4, 6)$ & B &  & $C_{15}(3)*$ & V & 1 & $C_{16}(2, 4, 6, 8)*$ & V & 1\\
$C_{8}(1, 3, 4)$ & B & 1 & $C_{12}(1, 3, 5, 6)$ & B & 1 & $C_{15}(5)*$ & V & 1 & $C_{16}(1, 2, 3, 4, 5)$ & V & 1\\
$C_{8}(1, 2, 3, 4)$ & V & 1 & $C_{12}(1, 4, 5, 6)$ & V & 1 & $C_{15}(1, 5)$ & N &  & $C_{16}(1, 2, 3, 4, 6)$ & V & 1\\
\cline{0-2}
$C_{9}(3)*$ & V & 1 & $C_{12}(2, 3, 4, 6)$ & V & 1 & $C_{15}(3, 5)$ & N &  & $C_{16}(1, 2, 3, 6, 8)$ & B & 1\\
$C_{9}(1, 3)$ & B &  & $C_{12}(1, 2, 3, 4, 5)$ & B &  & $C_{15}(3, 6)*$ & V & 1 & $C_{16}(1, 2, 3, 7, 8)$ & B & \\
$C_{9}(1, 2, 3)$ & V & 1 & $C_{12}(1, 2, 3, 4, 6)$ & V & 1 & $C_{15}(1, 2, 3)$ & B &  & $C_{16}(1, 2, 4, 5, 8)$ & B & \\
$C_{9}(1, 2, 4)$ & B &  & $C_{12}(1, 2, 3, 5, 6)$ & B &  & $C_{15}(1, 3, 5)$ & B &  & $C_{16}(1, 2, 4, 6, 7)$ & N & \\
$C_{9}(1, 2, 3, 4)$ & V & 1 & $C_{12}(1, 2, 4, 5, 6)$ & B & 1 & $C_{15}(1, 3, 6)$ & V & 1 & $C_{16}(1, 2, 4, 6, 8)$ & V & 1\\
\cline{0-2}
$C_{10}(2)*$ & V & 1 & $C_{12}(1, 3, 4, 5, 6)$ & B & 1 & $C_{15}(1, 4, 6)$ & N &  & $C_{16}(1, 2, 4, 7, 8)$ & B & 1\\
$C_{10}(5)*$ & V & 1 & $C_{12}(1, 2, 3, 4, 5, 6)$ & V & 1 & $C_{15}(3, 5, 6)$ & V & 1 & $C_{16}(1, 2, 6, 7, 8)$ & V & 1\\
\cline{4-6}
$C_{10}(1, 4)$ & N &  & $C_{13}(1, 3)$ & B &  & $C_{15}(1, 2, 3, 6)$ & N &  & $C_{16}(1, 3, 4, 5, 7)$ & N & \\
$C_{10}(2, 4)*$ & V & 1 & $C_{13}(1, 5)$ & V & 1 & $C_{15}(1, 2, 3, 7)$ & B & 1 & $C_{16}(1, 3, 5, 7, 8)$ & B & 1\\
$C_{10}(2, 5)$ & B &  & $C_{13}(1, 2, 4)$ & B &  & $C_{15}(1, 2, 5, 6)$ & B &  & $C_{16}(1, 2, 3, 4, 5, 6)$ & V & 1\\
$C_{10}(1, 2, 3)$ & V & 1 & $C_{13}(1, 2, 5)$ & B &  & $C_{15}(1, 3, 4, 5)$ & B & 1 & $C_{16}(1, 2, 3, 4, 5, 7)$ & B & 1\\
$C_{10}(1, 2, 4)$ & V & 1 & $C_{13}(1, 3, 4)$ & B & 1 & $C_{15}(1, 3, 4, 6)$ & B & 1 & $C_{16}(1, 2, 3, 4, 5, 8)$ & V & 1\\
$C_{10}(1, 2, 5)$ & B &  & $C_{13}(1, 2, 3, 4)$ & V & 1 & $C_{15}(1, 3, 5, 6)$ & B &  & $C_{16}(1, 2, 3, 4, 6, 8)$ & V & 1\\
$C_{10}(1, 3, 5)$ & B &  & $C_{13}(1, 2, 3, 5)$ & V & 1 & $C_{15}(1, 4, 5, 6)$ & V & 1 & $C_{16}(1, 2, 3, 4, 7, 8)$ & B & \\
$C_{10}(1, 4, 5)$ & V & 1 & $C_{13}(1, 2, 3, 6)$ & B &  & $C_{15}(1, 2, 3, 4, 5)$ & V & 1 & $C_{16}(1, 2, 3, 5, 6, 7)$ & B & \\
$C_{10}(2, 4, 5)$ & V & 1 & $C_{13}(1, 2, 3, 4, 5)$ & V & 1 & $C_{15}(1, 2, 3, 5, 6)$ & B &  & $C_{16}(1, 2, 3, 5, 6, 8)$ & V & 1\\
$C_{10}(1, 2, 3, 4)$ & B &  & $C_{13}(1, 2, 3, 4, 5, 6)$ & V & 1 & $C_{15}(1, 2, 3, 5, 7)$ & V & 1 & $C_{16}(1, 2, 3, 5, 7, 8)$ & B & \\
\cline{4-6}
$C_{10}(1, 2, 3, 5)$ & B & 1 & $C_{14}(2)*$ & N &  & $C_{15}(1, 2, 4, 5, 7)$ & B &  & $C_{16}(1, 2, 4, 6, 7, 8)$ & V & 1\\
$C_{10}(1, 2, 4, 5)$ & V & 1 & $C_{14}(7)*$ & V & 1 & $C_{15}(1, 3, 4, 5, 6)$ & V & 1 & $C_{16}(1, 3, 4, 5, 7, 8)$ & B & 1\\
$C_{10}(1, 2, 3, 4, 5)$ & V & 1 & $C_{14}(1, 6)$ & N &  & $C_{15}(1, 2, 3, 4, 5, 6)$ & V & 1 & $C_{16}(1, 2, 3, 4, 5, 6, 7)$ & B & \\
\cline{0-2}
$C_{11}(1, 2)$ & B &  & $C_{14}(2, 4)*$ & V & 1 & $C_{15}(1, 2, 3, 4, 5, 7)$ & B & 1 & $C_{16}(1, 2, 3, 4, 5, 6, 8)$ & V & 1\\
$C_{11}(1, 3)$ & B &  & $C_{14}(1, 2, 5)$ & B &  & $C_{15}(1, 2, 3, 4, 6, 7)$ & B &  & $C_{16}(1, 2, 3, 4, 5, 7, 8)$ & B & 1\\
$C_{11}(1, 2, 3)$ & V & 1 & $C_{14}(1, 4, 6)$ & B &  & $C_{15}(1, 2, 3, 4, 5, 6, 7)$ & V & 1 & $C_{16}(1, 2, 3, 5, 6, 7, 8)$ & B & 1\\
\cline{7-9}
$C_{11}(1, 2, 4)$ & B &  & $C_{14}(1, 4, 7)$ & B &  & $C_{16}(4)*$ & N &  & $C_{16}(1, 2, 3, 4, 5, 6, 7, 8)$ & V & 1\\
$C_{11}(1, 2, 3, 4)$ & V & 1 & $C_{14}(1, 6, 7)$ & B & 1 & $C_{16}(8)*$ & V & 1 & & &\\
\hline
\hline 
\end{tabular}
\end{tiny}
\end{center}

%%%%%%%%%%%%%%%%%%%%%%%%%%%%%%%%%%%%%%%%%%%%%%%%%%%%%%%%%%%%%%%%%%%%%%%%%%%%%%%%%%%%%

%%%%%%%%%%%%%%%%%%%%%%%%%%%%%%%%%%%%%%%%%%%%%%%%%%%%%%%%%%%%%%


\begin{thebibliography}{99}

\bibitem{A}
M. Adamaszek, 
Splittings of independence complexes and the powers of cycles. 
Combin. Theory Ser. A {\bf 119} (2012), 1031--1047. 

\bibitem{BGM} E. Boros, V. Gurvich, M. Milani\v{c}, 
On CIS circulants. 
Discrete Math. {\bf 318} (2014), 78--95.
	
\bibitem{BHmusic} J.\ Brown, R.\ Hoshino,
Independence polynomials of circulants with an application to music.
Discrete Math. {\bf 309} (2009), 2292--2304.

\bibitem{Brown11} J.\ Brown, R.\ Hoshino, 
Well-covered circulant graphs.
Discrete Math. {\bf 311} (2011), 244--251.


\bibitem{Cook} D. Cook II,
Simplicial decomposability. 
J. Softw. Algebra Geom. {\bf 2} (2010), 20--23. 

\bibitem{CN}
D. Cook II, U. Nagel,
Cohen-Macaulay Graphs and Face Vectors of Flag Complexes.
SIAM J. Discrete Math. {\bf 26} (2012), 89-–101.

\bibitem{DE}
A. Dochtermann, A. Engstr\"om, 
Algebraic properties of edge ideals via combinatorial topology. 
Electron. J. Combin. {\bf 16} (2009), no. 2, 
Special volume in honor of Anders Bj\"orner, Research Paper 2, 24 pp. 

\bibitem{eastman}
B. Eastman,
Circulant - \LaTeX\ code. (2014) GitHub repository,
\url{https://github.com/brydon/circulant}

\bibitem{EV}
M. Estrada, R. Villarreal, 
Cohen-Macaulay bipartite graphs. 
Arch. Math. (Basel) {\bf 68} (1997), 124--128. 

\bibitem{FHVT} 
C. Francisco, A. Hoefel, A. Van Tuyl, 
EdgeIdeals: a package for (hyper)graphs. 
J. Softw. Algebra Geom. {\bf 1} (2009), 1--4. 

\bibitem{Mt} 
D.\ R.\ Grayson, M.\ E.\ Stillman,
Macaulay 2, a software system for research in algebraic geometry.
\url{http://www.math.uiuc.edu/Macaulay2/}.

\bibitem{HHbook}
J. Herzog, T. Hibi,
{\it Monomial ideals}.
Graduate Texts in Mathematics, 260. Springer-Verlag London, Ltd., London, 2011.

\bibitem{HH}
J. Herzog, T. Hibi, 
Distributive lattices, bipartite graphs and Alexander duality. 
J. Algebraic Combin. {\bf 22} (2005), 289--302. 

\bibitem{HHZ}
J.  Herzog, T. Hibi, X. Zheng, 
Cohen-Macaulay chordal graphs. 
J. Combin. Theory Ser. A {\bf 113} (2006), 911-–916. 

\bibitem{Hbook}
T. Hibi, 
{\it Algebraic combinatorics on convex polytopes.} 
Carslaw Publications, Glebe, 1992.

\bibitem{HHKO}
T. Hibi, A. Higashitani, K. Kimura, A.B. O'Keefe,
Algebraic study on Cameron-Walker graphs.
J. Algebra {\bf 422} (2015),  257--269.

\bibitem{HPhD} R. Hoshino,
Independence polynomials of circulant graphs.  
Ph.D. Thesis, Dalhousie University, 2007.

\bibitem{K}
D.N. Kozlov, Complexes of directed trees. 
J. Combin. Theory Ser. A {\bf 88} (1999), 112--122.

\bibitem{MMCRTY}
M. Mahmoudi, A. Mousivand, M. Crupi, G.  Rinaldo, N. Terai, S. Yassemi, 
Vertex decomposability and regularity of very well-covered graphs. 
J. Pure Appl. Algebra {\bf 215} (2011), 2473--2480. 

\bibitem{MV} S. Morey, R. Villarreal, 
Edge ideals: algebraic and combinatorial properties. 
{\it Progress in commutative algebra 1}, 85--126, de Gruyter, Berlin, 2012. 

\bibitem{MT} S. Moriyama, F. Takeuchi, 
Incremental construction properties in dimension two-shellability, 
extendable shellability and vertex decomposability. 
Discrete Math. {\bf 263} (2003), 295--296.

\bibitem{M}  R. Moussi, 
A characterization of certain families of well-covered circulant graphs. 
M.Sc. Thesis, St. Mary's University, 2012.

\bibitem{Mu}
S. Murai, 
Face vectors of two-dimensional Buchsbaum complexes. 
Electron. J. Combin. {\bf 16} (2009), no. 1, Research Paper 68, 14 pp. 
  
\bibitem{P}  M. Plummer, 
Well-covered graphs: a survey. 
Quaestiones Math. {\bf 16} (1993), 253--287. 

\bibitem{PB}
J.  Provan and L. Billera, 
Decompositions of simplicial complexes related to diameters of convex polyhedra.
Math. Oper. Res. {\bf 5} (1980), 576--594.

\bibitem{S}
J. A. W. Staples, 
On some subclasses of well-covered graphs. Ph.D. Thesis, Vanderbilt
University, 1975.

\bibitem{Sage}
W. A. Stein et al. Sage Mathematics Software (Version 6.2),
The Sage Development Team, 2014, http://www.sagemath.org.

\bibitem{T}
N. Terai, 
On $h$-vectors of Buchsbaum Stanley-Reisner rings. 
Hokkaido Math. J. {\bf 25} (1996), 137--148.


\bibitem{VVW} K.\ N.\ Vander Meulen, A. Van Tuyl, C. Watt,
Cohen-Macaulay Circulant Graphs,
Comm. Algebra {\bf 42} (2014), 1896--1910.

\bibitem{VT}
A. Van Tuyl, 
Sequentially Cohen-Macaulay bipartite graphs: vertex decomposability and regularity. 
Arch. Math. (Basel) {\bf 93} (2009), 451--459. 

\bibitem{V}
R.\ H. Villarreal, Cohen-Macaulay graphs. 
Manuscripta Math.  {\bf 66} (1990), 277--293.

\bibitem{Vbook} R. H. Villarreal,
{\it Monomial algebras.} 
Monographs and Textbooks in Pure and Applied Mathematics, {\bf 238}.
Marcel Dekker, Inc., New York, 2001.

\bibitem{W1978} D. Walkup, 
The Hirsch conjecture fails for triangulated 27-spheres.
Math. Oper. Res. {\bf 3} (1978), no. 3, 224--230. 

\bibitem{W}
R. Woodroofe, 
Vertex decomposable graphs and obstructions to shellability. 
Proc. Amer. Math. Soc. {\bf 137} (2009) 3235--3246. 
\end{thebibliography}
\end{document}